\documentclass[11pt]{amsart}
\usepackage[utf8]{inputenc}
\usepackage[centertags]{amsmath}
\usepackage{amsfonts}
\usepackage{amsthm}
\usepackage{newlfont}
\usepackage{amscd}
\usepackage{amsgen}
\usepackage{amssymb}
\usepackage{fullpage}
\usepackage{longtable}
\usepackage{listings}
\usepackage{easybmat}
\theoremstyle{plain}
\newtheorem{thm}{Theorem}[section]

\newtheorem{lemm}[thm]{Lemma}
\newtheorem{prop}[thm]{Proposition}
\newtheorem{cor}[thm]{Corollary}

\theoremstyle{definition}

\newtheorem{rem}[thm]{Remark}
\newtheorem{constr}[thm]{Construction}

\newtheorem{assumptions}[thm]{Assumption}

\numberwithin{equation}{section}

\theoremstyle{definition}

 \newcommand\PP{{\mathbb{P}}}

          \newcommand\oo{{\mathcal O}}

\title{Bilinkage in codimension $3$ and canonical surfaces of degree
$18$ in $\PP^5$}
\author{Grzegorz Kapustka and Micha{\l} Kapustka}
\begin{document}
\setcounter{tocdepth}{1}
\begin{abstract} We study the behavior of the bilinkage process in codimension $3$.
In particular, we construct a smooth canonically embedded and linearly normal surface of
general type of degree $18$ in $\mathbb{P}^5$; this is probably the
highest degree
  such a surface may have. Next, we apply our construction to find a geometric description of Tonoli Calabi--Yau threefolds in $\PP^6$.
\end{abstract}
\maketitle
\section{Introduction} Let $S$ be a minimal surface of general type defined over
the field of complex numbers. Then, by the inequality of Noether
and Bogomolov--Miyaoka--Yau, we have $$2\chi(\oo_{S})-6\leq K_S^2\leq
9\chi(\oo_S).$$
On the other hand, if we assume that the canonical system of $S$ gives a
birational map, then by the Castelnuovo inequality we deduce $3\chi(\oo_S)-10\leq
K_S^2.$ Note that we know from \cite{Bombieri} that $5K_S$ always gives a birational morphism for surfaces of general type. In this context, it is a natural problem (cf.~\cite{Ashikaga}, \cite{Catanese}) to
construct surfaces of general type with birational canonical map in the
  range
$3\chi(\oo_{S})-10\leq K_S^2\leq 9\chi(\oo_S)$. Many works are related to this
problem \cite{Ashikaga}, \cite{HirzebruchVandeVen}, \cite{Sommese},
however, the part with $\chi(\oo_S)\leq 7$ seems out of reach with those methods.
The general surface of general type with $\chi(\oo_S)=7$ and $h^1(\oo_S)=0$
should admit a birational canonical map to $\mathbb{P}^5$.
The image of such a map is a subcanonical surface of codimension $3$ in $\PP^5$.

On the other hand, it was proven in \cite{Walter} that
  submanifolds  $X\subset \PP^N$ of codimension $3$ in
projective spaces with $N-3$ not divisible by $4$ that are subcanonical are
  Pfaffian, i.e.~their ideal sheaf admits a Pfaffian resolution of the form \begin{equation}\label{pfaff-resol} 0\to\mathcal{O}_{\mathbb{P}^N}(-2s-t)\to E^{\ast}(-s-t)\xrightarrow{\varphi}
E(-s)\xrightarrow{\psi} \mathcal{I}_X\to 0 \end{equation} where $E$ is a vector bundle of odd rank and $s,t\in \mathbb{Z}$.
The study of codimension $3$ manifolds is reduced
in this way to the study of the Hartshorne--Rao modules of
  submanifolds. However, complicated algebraic problems appear when we want to
classify such modules (see \cite{CYTon1}).
Catanese \cite{Catanese} applied the Pfaffian
construction in order to construct canonically embedded surfaces in $\PP^5$ and
found constructions
of surfaces with $K_S^2 \leq 17$. Later, in his thesis \cite{Tonoli}, Tonoli
constructed Calabi--Yau threefolds in $\PP^6$, but
  only found examples of
degree $\leq 17$. Since the Pfaffian construction becomes more and more
complicated when the degree increases, it is natural to
  ask whether there are any canonical surfaces of degree $\geq 18$ in $\PP^5$. Our main result is the
following:
\begin{thm}\label{0}
There exists a surface of general type with $K^2=18$, $p_g=6$, $q=0$
whose canonical map is an isomorphism onto its image.
\end{thm}
We describe the construction of such surfaces in Section \ref{sec-canonical18} concluding with Theorem \ref{gen type surfaces of
deg 18}.
We expect, by \cite{CYTon1}, that this is the highest degree of
  such a
canonically embedded surface in $\PP^5$.
In fact, in Theorem \ref{gen type surfaces of deg 18}, we find an explicit description of a $20$-dimensional
   subfamily of the
at least
  36-dimensional family of degree $18$ canonical surfaces in $\PP^5$.
Having our existence result, it is a natural problem (see \cite{CYTon1}) to find a Pfaffian
resolution for a general canonical surface of degree $18$ in $\mathbb{P}^5$.

The idea of the proof of Theorem \ref{0} is to construct a special
  bilinkage. Recall that
  the relation of \emph{linking} (or equivalently \emph{liaison}) was introduced in \cite{PS}. Two closed subschemes
$V$, $W$ of $\PP^N$ are algebraically linked by a subscheme $X$
if they are equidimensional without embedded
components and $X$
is a complete intersection containing them such that $\mathcal{I}_{W|\PP^N}/\mathcal{I}_{X|\PP^N}=\mathcal{H}om(\oo_V,\oo_X)$ and $\mathcal{I}_{V|\PP^N}/\mathcal{I}_{X|\PP^N}=\mathcal{H}om(\oo_W,\oo_X)$. When additionally
  $W$ and $V$ do not have common components (this is the situation we are interested in) then
they are linked if $V\cup  W=X$. We say that two irreducible varieties of the same dimension are \emph{bilinked} if they are linked in two
  steps, i.e.~there exists a scheme $T$ such that $T$ is linked with both $W$ and $V$.
We shall also need another point of view
  on bilinkage: the notion of \emph{generalized divisors} introduced by Hartshorne. With the notation as above with $V$ and $W$ bilinked we see that $V$ and $W$ are of codimension
one in $X$
  and thus can be seen
  as ``divisors'' in $X$. Hartshorne
   \cite[\textsection 2]{Hartshornebil}
  generalizes the usual notion of Cartier divisor and linear equivalence of divisors ``$\simeq$''  (in our case $X$ is very singular
  so $V$, $W$ are not Cartier
  divisors) in order to obtain a relation
$V\simeq W+ nH$ where $n\in \mathbb{Z}$ and $H$ is
  a hyperplane section of $X\subset \PP^N$. 

  Let us  now describe our construction of the surface in Theorem \ref{0}. We first take a special central projection to
$\mathbb{P}^5$
of the image $V_9$ of the third Veronese embedding of $\mathbb{P}^2$ in
$\mathbb{P}^9$ and perform a bilinkage.
More precisely, we find a special $\mathbb{P}^3\subset \mathbb{P}^9$ such that the image $D_9\subset\mathbb{P}^5$ of
$V_9\subset\mathbb{P}^9$ by the projection centered in this $\mathbb{P}^3$ is smooth and
contained in the complete intersection of two cubics. Note here that
  this image is contained in a
single cubic for a generic projection (cf.~\cite[Rem.~5.5]{GKapustkaprojdelpezzo}).
Then we perform a bilinkage of $D_9$ through the intersection of these two
  cubics,
obtaining a special smooth canonically embedded general type surface of degree
$18$ (cf.~\cite{CYTon1}).
In Proposition \ref{lem can surfaces are bilinked to del pezzo}, we show that all
the known examples of canonical surfaces in $\PP^5$ from \cite{Catanese}
can be obtained
  via
  that bilinkage construction.
We believe that our construction can be applied in a more general
classification problems concerning submanifolds of codimension $3$.

Since the method of bilinkage
  works better than the Pfaffian construction in
the case of surfaces, we
  apply it in Section \ref{sec-Ton} to study
Tonoli
 Calabi--Yau threefolds in $\PP^6$. Those threefolds were constructed
   in \cite{Tonoli}
(cf.~\cite{SchreyerTonoli}, \cite{CYTon1}) by using the
Pfaffian resolution (\ref{pfaff-resol}).
Our first result, Proposition \ref{prop bilinkage dP 7 and CY 16}, says that
starting from del Pezzo threefolds of degree $d\leq 7$ in $\PP^6$ we obtain
families of
Tonoli Calabi--Yau threefolds of degree $d+9$ by performing the
bilinkage construction through the intersection of two cubics
(cf.~\cite{HulekRanestad}).
In the remaining degree $17$, there are three families of Calabi--Yau threefolds that we call
after Tonoli of type $k=8$, $k=9$, and $k=11$.
The corresponding degree $8$ del Pezzo threefold is the double Veronese
embedding of $\PP^3$ projected to $\PP^6$.
  As before, we can find a special center of projection such that the
image of the del Pezzo threefold of degree 8 is smooth and contained in a
  three-dimensional space of cubics. Note that it is
contained in no
  cubic for a general projection.
So we can perform a bilinkage and
  its result is a natural
  degeneration of
the degree $17$ Tonoli family of type $k=9$.
  Note that the examples of type $k=8$
  and $k=11$ cannot be
constructed by bilinkage. This shows that the construction that we propose in \cite[Thm.~1.3]{CYTon1}, by operations on vector
  bundles from the Pfaffian
  resolution,
is, in this context, a strict generalization of the one using bilinkages.
We close Section \ref{sec-Ton}
  with the
construction of a singular degree $18$ threefold in $\PP^6$ birational to a
Calabi--Yau threefold.

In
  Section
  \ref{subsec bilinkages} we study the relation between the Pfaffian resolutions \ref{pfaff-resol} of two bilinked subvarieties of codimension $3$.
Finally, in Section \ref{sec-unpr}, we discuss relations between constructions by bilinkage and by
  unprojection, finding that the former are more
general in our situation. This confirms the general Reid philosophy about the relation between these
constructions. As a result, we analyze an example of
non-Gorenstein unprojection that should be of independent interest.
\section*{Acknowledgments} We would like to thank Ch.~Okonek for all his
advice
and support, and
  J.~Buczy\'nski, S.~Cynk, L.~Gruson, A.~Kresch, A.~Langer,
P.~Pragacz for comments and discussions.
The use of Macaulay 2 was essential to guess the geometry.
The project was supported by MNSiW, N N201 414539 and by the Forschungskredit of
the University of Zurich.
  \section{Preliminaries} \label{sec bilink}
We shall apply the following construction to relate a given del Pezzo
surface $F\subset\mathbb{P}^5$ (resp.~del Pezzo threefold)
  to a
  surface $X\subset\mathbb{P}^5$ of
general type (resp.~Calabi--Yau threefold).
\begin{constr} \label{relation by bilinkage and deformation surf}
We write
$$\mathbb{P}^n\supset F \rightrightarrows X' \rightsquigarrow  X \subset
\mathbb{P}^n,  $$
where $F\rightrightarrows X'$ means that $F$ and $X'$ are bilinked and
$X'\rightsquigarrow X$ means that $X'$ is a degeneration of
  $X$, i.e.~there
is a proper flat family
over a disc such that $X'$ is its special element and $X$ a general one. Then we
say that $X$ is constructed from $F$ by the bilinkage construction.
\end{constr}
Before we study the possible applications of this construction, let us consider bilinkages of Pfaffian varieties in general.
\subsection{Bilinkages of Pfaffians} \label{subsec bilinkages}
Let us make some useful remarks on the construction of bilinkages between Pfaffian varieties by relating vector bundles defining them.  More precisely,
we aim at proving that under some
  assumptions, if two bundles
  $E$ and $F$ of odd rank differ by a sum of line bundles then the Pfaffian varieties
associated to general sections of their twisted wedge squares are in the same complete intersection biliaison class.

Let $X\subset \mathbb{P}^N$ be a Pfaffian variety defined by a section $\varphi
\in H^0(\bigwedge^2 E \otimes \mathcal{O}_{\mathbb{P}^N}(t))$ for some vector
bundle $E$ of rank $2r+1$ for some $r\in \mathbb{N}$ and $t\in \mathbb{Z}$.
Denote $s=c_1(E)+rt$.

The map $\varphi$ in the Pfaffian resolution (\ref{pfaff-resol}) is identified with the section
$$\varphi \in H^0(\mathbb{P}^N, {\textstyle\bigwedge}^2 E \otimes
\mathcal{O}_{\mathbb{P}^N}(t)),$$
and $\psi$ is the map
$$E(-s)\rightarrow \mathcal{I}_{X}=\mathrm{Im} (\psi)\subset {\textstyle\bigwedge}^{2r+1} E
\otimes \mathcal{O}_{\mathbb{P}^N}(rt-s)= \mathcal{O}_{\mathbb{P}^N}$$
defined as the wedge product with the
  $r$-th divided power of $\varphi$:
$$\frac{1}{r!} (\varphi\wedge \varphi \wedge \dots \wedge \varphi) \in
H^0(\mathbb{P}^N, {\textstyle\bigwedge}^{2r} E \otimes \mathcal{O}_{\mathbb{P}^N}(rt)).$$

\begin{assumptions}\label{assumptions for biliaison}
 $H^1(E^*(l))=0$ for $l\in \mathbb{Z}$.
\end{assumptions}
Observe that Assumption \ref{assumptions for biliaison} is satisfied when $E$ is obtained as the kernel of a surjective
map between decomposable bundles.

Under Assumption \ref{assumptions for biliaison} on $E$ we claim
that every hypersurface of degree
$d$ containing $X$ is defined as a Pfaffian hypersurface given by a section of
the bundle $\bigwedge^2 (E\oplus \mathcal{O}_{\mathbb{P}^N}(d-s-t))\otimes
\mathcal{O}_{\mathbb{P}^N}(t)$ of even rank $2r+2$.
Indeed, we can split the Pfaffian sequence into two short exact sequences:
\begin{gather*}
  0\to\mathcal{O}_{\mathbb{P}^N}(-2s-t)\to E^{\ast}(-s-t)\to F\to 0,\\
0 \to F\to E(-s)\xrightarrow{\psi} \mathcal{I}_X\to 0,
\end{gather*}
for some sheaf $F$.
Taking the cohomology of the second, we obtain an exact sequence
$$H^0(E(d-s))\to H^0(\mathcal{I}_X(d)) \to H^1(F(d)).$$
  On the other hand, by the first exact sequence we
  have
$$ H^1(E^{\ast}(d-s-t))\to H^1(F(d))\to H^2(\mathcal{O}_{\mathbb{P}^N}(d-2s-t)),$$
and it follows
  from the assumption on $E$ and the fact that $N\geq 3$ that $H^1(F(d))=0$.
We hence have a surjection
$$H^0(E(d-s))\to H^0(\mathcal{I}_X(d))$$
induced by $\psi$. Thus every hypersurface of degree $d$ in the ideal of $X$ is
identified with a section
$$\mathbf{s}\wedge\varphi^{(r)}\in H^0(\mathbb{P}^N, {\textstyle\bigwedge}^{2r+1} E \otimes
\mathcal{O}_{\mathbb{P}^N}(d-s+rt)),$$
for some section
$\mathbf{s} \in H^0(E(d-s))$.
It is now enough to observe that
$$\mathcal{O}_{\mathbb{P}^N}(d)= {\textstyle\bigwedge}^{2r+1} E \otimes
\mathcal{O}_{\mathbb{P}^N}(d-s+rt)= {\textstyle\bigwedge}^{2r+2}(E\oplus
\mathcal{O}_{\mathbb{P}^N}(d-s-t))\otimes \mathcal{O}_{\mathbb{P}^N}((r+1)t)$$
and the section $\mathbf{s}\wedge\varphi^{(r)}$ corresponds to the Pfaffian of
the section
$$(\varphi,\mathbf{s}) \in H^0({\textstyle\bigwedge}^2(E\oplus
\mathcal{O}_{\mathbb{P}^N}(d-s-t))\otimes
\mathcal{O}_{\mathbb{P}^N}(t))=H^0({\textstyle\bigwedge}^2 E \otimes
\mathcal{O}_{\mathbb{P}^N}(t)) \oplus H^0(E \otimes
\mathcal{O}_{\mathbb{P}^N}(d-s))$$
under the above identification.

\begin{lemm}\label{lem bilinkage 1} Let $X$, $E$, $r$, $s$, $\varphi$ be as
above and let $E$ satisfy Assumption \ref{assumptions for biliaison}.
Assume that $X$ is contained in two hypersurfaces $H_{d_1}$ and $H_{d_2}$ of
degree $d_1$ and $d_2$ respectively. Let $\mathbf{s}_i$ be the
section of $H^0(E(d_i-s))$ corresponding to $H_{d_i}$ for $i=1,2$. Then
$H_{d_1}\cap H_{d_2}$ is a codimension $2$ complete intersection
if and only if the section $(\varphi,\mathbf{s}_1,\mathbf{s}_2,\mathbf{l})$
in the decomposition
\begin{multline*}
H^0({\textstyle\bigwedge}^2(E\oplus \mathcal{O}_{\mathbb{P}^6}(d_1-s-t)\oplus
\mathcal{O}_{\mathbb{P}^6}(d_2-s-t))(t))\\
=H^0({\textstyle\bigwedge}^2E(t))\oplus
H^0(E(d_1-s))\oplus
H^0(E(d_2-s))\oplus H^0(\mathcal{O}_{\mathbb{P}^6}(d_1+d_2-2s-t))
\end{multline*}
defines a codimension $3$ Pfaffian variety for general $\mathbf{l}\in
H^0(\mathcal{O}_{\mathbb{P}^6}(d_1+d_2-2s-t))$. Moreover, if $Y$ is a Pfaffian
variety defined by a section
$(\varphi,\mathbf{s}_1,\mathbf{s}_2, \mathbf{l})$ then $Y$ is bilinked to $X$ via the
intersection of the two hypersurfaces $H_{d_1}$
and $H_{d_2}$.
\end{lemm}
\begin{proof}
Assume that $T=H_{d_1}\cap H_{d_2}$ is a codimension 2 complete intersection.
Let us choose $\mathbf{l}\in H^0(\mathcal{O}_{\mathbb{P}^6}(d_1+d_2-2s-t))$ such
that $T\cap\{\mathbf{l}=0\}$ is of codimension 3.
We can now easily check that at $\lambda=0$ the degeneracy loci of the
sections
$(\varphi,\mathbf{s}_1,\lambda \mathbf{s}_2, \mathbf{l})$ degenerate to a subvariety of
$X\cup (Y\cap\{\mathbf{l}=0\})$, hence the
general element of the family has codimension 3 as expected.
By simple base change we
  find that the map associated to
$(\varphi,\mathbf{s}_1,\mathbf{s}_2,\frac{1}{\lambda}\mathbf{l})$ degenerates
along a codimension 3 Pfaffian variety for general $\lambda$.

  Conversely, assume that a section
$(\varphi,\mathbf{s}_1,\mathbf{s}_2,\mathbf{l})$ defines a codimension 3
variety.
  Consider two sections
$$g_1=\mathbf{s}_1\wedge \mathbf{s}_2 \wedge \underbrace{\varphi \wedge \dots \wedge
\varphi}_{r-1} \in H^0({\textstyle\bigwedge}^{2r+3}(E\oplus
\mathcal{O}_{\mathbb{P}^6}(d_1-s-t)\oplus
\mathcal{O}_{\mathbb{P}^6}(d_2-s-t))((r+1)t) )$$ and
$$g_2=\mathbf{l} \otimes(t \wedge \underbrace{\varphi \wedge \dots \wedge
\varphi}_{r})\in H^0({\textstyle\bigwedge}^{2r+3}(E\oplus
\mathcal{O}_{\mathbb{P}^6}(d_1-s-t)\oplus
\mathcal{O}_{\mathbb{P}^6}(d_2-s-t))((r+1)t) ).$$
It is a simple exercise in linear algebra to check on
  the fibers of this bundle that
these two sections are proportional on $H_{d_1}\cap H_{d_2}$,
 i.e.
 $\mathbf{s}_i\wedge \underbrace{\varphi \wedge \dots \wedge \varphi}_{r}$
  vanishes for
$i=1,2$.
The ratio between these two sections defines a rational function $g$ on
$H_{d_1}\cap H_{d_2}$. Observe that the function $g+1$ vanishes only along the
Pfaffian variety defined by
  $(\varphi,\mathbf{s}_1,\mathbf{s}_2,\mathbf{l})$,
which is of codimension 3 by assumption. It follows that $H_{d_1}\cap H_{d_2}$
is of codimension 2.

  To prove the last statement of the lemma, let $Y$ be
  the Pfaffian variety
defined by the section
   $(\varphi,\mathbf{s}_1,\mathbf{s}_2,\mathbf{l})$. In
particular, $Y$ is of codimension 3 and  $T=H_{d_1}\cap H_{d_2}$ is a
complete intersection of two hypersurfaces.
Now, $Y$ defines a generalized divisor in the sense of \cite{Hartshornebil}. We
claim that $Y$ is linearly equivalent as a generalized
  divisor to $X+H$, where
$H$ is the restriction of the hyperplane section
  to~$T$.
Indeed,
   $g-1$  is a rational function on
$H_{d_1}\cap H_{d_2}$
which defines $Y-X-H$. By
  the definition of biliaison, it follows that $Y$ is related
to $X$ by a biliaison of height 1.
Finally, by
  \cite[Prop. 4.4]{Hartshornebil}, this means that $Y$ is bilinked to
$X$.
\end{proof}
\section{Degree $9$ del Pezzo surface and degree $18$ canonically embedded
surfaces}\label{sec-canonical18}
The analogy discussed in \cite{CYTon1} suggests that one might try
to construct
a canonically embedded surface of general type of degree $18$ in $\mathbb{P}^5$
if
  one finds an
appropriate description of
a del Pezzo surface of degree 9 in $\mathbb{P}^5$. In this section we collect information
on such del Pezzo
  surfaces and next present a construction of canonically
embedded
surfaces of general type of degree $18$.
\subsection{Del Pezzo surface of degree 9}
Recall that a del Pezzo surface of degree 9 is just $\mathbb{P}^2$ and its
anticanonical embedding is the image $V_9$ of the triple Veronese embedding.
We shall denote by $D^{\Lambda}_9\subset \mathbb{P}^5=:\mathbb{P}(W)$ the surface obtained as the image of the
projection of the image of this embedding
from a general $3$-dimensional linear subspace $\Lambda\subset\mathbb{P}^9$.
Let $D=D^{\Lambda}_9$ for some general $3$-dimensional linear subspace $\Lambda\subset\mathbb{P}^9$. Our aim is to understand the module
  $M:=\bigoplus_{k=0}^{\infty}
H^1(\mathcal{I}_{D}(k+2))$, the shifted Hartshorne--Rao module of $D$. From \cite[Lem.~4.1]{CYTon1}, we know that the Hilbert function of the Hartshorne--Rao module of $D$ has values
  $(0,4,7,0, \dots )$ starting from grade $0$.
We can thus write $M=M_{-1}\oplus M_{0}=H^1(\mathcal{I}_D(1))\oplus
H^1(\mathcal{I}_{D}(2))$.

By working out a random example in Macaulay2, we
  can prove that $M$ is generated in
degree $-1$. Moreover, the minimal resolution of $M$
  is
$$
\vbox{%
\halign{&\hfil\,$#$\,\hfil\cr
0\leftarrow M\leftarrow 4R(1) \leftarrow 17R&&18 R(-1)&& 4R(-2)\ \cr
&\vbox to 10pt{\vskip 0pt\hbox{$\nwarrow$}\vss}&\oplus
&\longleftarrow & \oplus & \vbox to 10pt{\vskip 6pt\hbox{$\nwarrow$}\vss} \cr
&& 29 R(-2) && 80 R(-3)&&
81R(-4)\leftarrow 38 R(-5)\leftarrow 7 R(-6)\leftarrow 0 \cr
}}
$$
Trying to extend
  \cite[Thm.~1.3]{CYTon1} to the case of del Pezzo surfaces of degree 9, we should look for Pfaffian varieties associated to bundles in $\mathrm{Ext}^1(2\mathcal{O}_{\mathbb{P}^5}, Syz^1(M))$.
However, for modules corresponding to general del Pezzo surfaces $D^{\Lambda}_9$ it is hard to find any such bundle for which a Pfaffian variety would exist. In particular,  even the bundle $Syz^1(M)\oplus
2\mathcal{O}_{\mathbb{P}^5}$ has no twisted
skew self-map defining a Pfaffian variety.
For this reason, instead of trying to find a special element in $\mathrm{Ext}^1(2\mathcal{O}_{\mathbb{P}^5}, Syz^1(M))$ for $M$ being the shifted Hartshorne--Rao module of a general del Pezzo surface of degree 9 in $\mathbb{P}^5$,
we look for
  a special smooth projected del Pezzo surface of degree 9  for which the
Betti table of the resolution of its Hartshorne--Rao module has
different shape from the generic one.

\begin{prop}\label{prop special Lambda}
There exists a $3$-dimensional linear subspace $\Lambda \subset \mathbb{P}^9$ such that the projected del Pezzo
surface $D^{\Lambda}_9\subset \mathbb{P}^5$ is smooth and is contained in a
complete intersection $Y$ of
two smooth cubic hypersurfaces and such that the singular locus of $Y$ consists
of 60 isolated singularities.
\end{prop}
\begin{proof}
Let $V_9$ be the del Pezzo surface of degree 9 obtained as the image of the
following Veronese
  embedding:
$$\mathbb{P}^2\ni (x:y:z)\mapsto (x^3 : y^3 : z^3 : 3x^2y : 3xy^2 : 3x^2z :
3xz^2 : 3y^2z : 3yz^2 : 6xyz) \in \mathbb{P}^9.$$
Let us denote the corresponding coordinates in $\mathbb{P}^9$ by
  $(a_0, \dots, a_9)$.

Recall that the ideal of the del
Pezzo surface of degree $9$ in $\mathbb{P}^9$ is defined by the $2\times 2$
minors of the Catalecticant
  matrix
$$A=\left[\begin{array}{cccccc}
3a_0&a_4&a_6&2a_3&2a_5&a_9\\
a_3&3a_1&a_8&2a_4&a_9&2a_7\\
a_5&a_7&3a_2&a_9&2a_6&2a_8
 \end{array}\right].$$
It is also known that the $3\times 3$ minors of $A$ define the secant locus (i.e. the closure of the union of all secant lines) of
the del Pezzo surface. Moreover, the secant locus is
of codimension $4$ in $\mathbb{P}^9$.

By the probabilistic method, one can easily construct in positive characteristic
an example of a three-dimensional projective space $\Lambda_0$ disjoint from the
secant locus such that the projection is contained in a pencil
of cubics.
An example of a matrix defining a projection from such a $\Lambda_0$ in characteristic $17$ is
the following:
$$N_0=\left[\begin{array}{cccccc}
0  &  0 &  -1 &   -2 &0  & 0\\
        1  &  0 & -2  &  0 & 0 & 0  \\
        0   &  0  & -1  &  0  & 0 &-1 \\
        1 &  0 & 0  &   -1 & 0 & 0  \\
        1   &  1  & 2    & 0 & 0 & 0  \\
        0    & 0 & 0  &   0 &2 & 0 \\
        -1 &0 & 1   &  1  & 1 & 1  \\
        1 & 0 & 0  &   0 &  0 & 0  \\
        1  & 0 & 0   & 0 &1 & 0  \\
        -1  &  0  & 0  &   1 &  0 & 0
\end{array}\right].$$

We check using Macaulay2 that the projection of $V_9$ from $\Lambda_0$ is
contained in a complete intersection $Y$ of two smooth cubics $C_1$ and $C_2$
such that $Y$ has 60 distinct singular points.

We shall show that $\Lambda_0$ lifts to characteristic $0$ to some $\Lambda$
such that the projection of $V_9$ from $\Lambda$ is contained
  in two cubics
  that  specialize to $C_1$ and $C_2$.
First observe that a projection from $\Lambda$ corresponds to a linear map
$\mathbb{C}^{10}\to \mathbb{C}^6$, hence a $10 \times 6$ matrix $N$ with complex
entries.
This projection composed with the Veronese embedding is a map
$$\varphi:\mathbb{P}^2 \to \mathbb{P}^6$$
defined by a base point free linear system of cubics, described by $N$. If we
further compose $\varphi$ with the triple Veronese embedding $\psi:
\mathbb{P}^6\to \mathbb{P}^{55}$,
then the dimension of the space of cubics containing $\Pi_{\Lambda}(V_9)$ is
equal to the codimension of the span of the image
$\psi\circ\varphi(\mathbb{P}^2)$ in $\mathbb{P}^{55}$.
On the other hand, it is clear that $\psi\circ\varphi$ factors through a 9-tuple
Veronese embedding
   $\mathbb{P}^2\to\mathbb{P}^{54}$ and a linear map
$L_{N}:\mathbb{P}^{54}\to \mathbb{P}^{55}$.
It follows that the image of the projection $\Pi_{\Lambda}(V_9)$ is always
contained in a cubic hypersurface. Moreover, the projection is contained in a
  two-dimensional space if and only if
$L_N$ has non-maximal rank. Observe that, following the above description, $L_N$
can be written explicitly as a $55\times 56$ matrix depending on the entries of
$N$. If we now consider a
  60-dimensional vector space $V$ parametrizing
  matrices $N$, then we obtain a $55\times 56$ matrix $L$ with entries being cubic
polynomials on $V$.
Denote by $\Gamma$ the degeneracy locus of $L$. It is a subvariety of $V$ of
codimension $\leq 2$. We can now proceed to
  describe an explicit lifting to characteristic $0$ of the constructed
case over $\mathbb{F}_{17}$.

Consider any lift $N'_0$ of $N_0$ to $\mathbb{Z}$ and a random line $l$ in $V$
passing through $N_0$. More precisely, we choose $l$ by choosing a
parametrization $\mathcal{N}: \mathbb{C}\ni \lambda \mapsto N'_0+\lambda N_1\in l$ with
random $N_1$, for example
$$\mathcal{N}:=\left[\begin{array}{cccccc}
0  &   \lambda &  -2\lambda-1 &   -2 &0  & 0\\
        1  &   2\lambda & -2  &  -\lambda & 0 & 0  \\
        0   &  0  & -1  &  0  & 2\lambda &-1 \\
        \lambda+1 &  2\lambda & 0  &   -1 & 0 & 0  \\
        1   &  1  & 2    & -\lambda & 0 & 0  \\
        0    & -\lambda & 0  &   -2\lambda &2 & 2\lambda \\
        -2\lambda-1 &2\lambda & 1   &  1  & 1 & 1  \\
        2\lambda+1 & -\lambda & 0  &   0 &  0 & 0  \\
        \lambda+1  & -2\lambda &-\lambda   & -2\lambda &1 & 0  \\
        -1  &  0  & 0  &   1 &  0 & 0
\end{array}\right].$$

We can now easily compute in Macaulay2 the Smith normal form $L_{SNF}$ of the
matrix $L_\mathcal{N}$ restricted to the line $l$ (i.e. over
$\mathbb{Q}[\lambda]$). It is a matrix with
  entries  polynomial in
$\lambda$. More precisely, in our specific case,
 one entry is a polynomial $p$ of degree 150 with
 integer coefficients,
whereas the remaining diagonal entries are 1. By definition, $L_\mathcal{N}$
has non-maximal rank if and only $L_{SNF}$ has non-maximal rank. The degeneracy locus
of $L_{SNF}$ is clearly defined by the vanishing of $p$. We also check that the
reduction mod 17 of this polynomial is also of degree 150.
It follows by the Valuative Criterion for Properness that there exists a number
field $K$ and a prime $\mathfrak{p}$ in its ring of integers $O_K$ with
$O_K/\mathfrak{p}\cong \mathbb{F}_{17}$ such that
$N_0$ is the specialization of an $O_{K,\mathfrak{p}}$ valued point of $\Gamma$.
In our case,
this can be
  shown explicitly. Indeed, the polynomial of degree 150
decomposes into two irreducible (over $\mathbb{Q}[\lambda]$) polynomials $P_1$
and $P_2$ with integer coefficients and of degrees 60 and 90 respectively.
We check that $P_1$ reduced mod 17 has a root
   $\lambda=0$. We can hence
consider the number field $K=\mathbb{Z}[\lambda]/(P_1)$ and the prime ideal
generated by $(17,\lambda)$ in $O_K$.
It is clear that
  the projection $\Pi$ defined by the matrix $L_{\mathcal{N}}$ with $\lambda$
being any root of
   $P_1$ maps $V_9$ to a variety contained in two
cubics.

Observe now that the two cubics containing  $\Pi(V_9)$ are 
computed universally (with parameter $\lambda$)
  when computing the Smith normal form. More precisely, from the
algorithm, we obtain matrices $\mathcal{S}_1$, $\mathcal{S}_2$ invertible over
$\mathbb{Q}[\lambda]$ such that $\mathcal{S}_1 L_{\mathcal{N}}
\mathcal{S}_2=L_{SNF}$. In particular, the
columns of $\mathcal{S}_2$
corresponding to
  the vanishing columns of $L_{SNF}$ define cubics containing the
image $\Pi(V_9)$.
We check easily that the ideal generated by these two cubics specializes via our
specialization map to the ideal generated by the two
  cubics, computed over
$\mathbb{F}_{17}$.
Moreover, since the kernel of the projection over $\mathbb{F}_{17}$ is disjoint
from the secant locus, this is also the case for the lifted projection.
It follows that there exists a lift of the projection found over
$\mathbb{F}_{17}$ such that the image of $V_9$ is smooth and contained in a
complete intersection $Y$ of two smooth cubics and such that $Y$ has
singularities not worse than 60
isolated singular points.
\end{proof}

\begin{rem} Note that the computation of
$c_2(\mathcal{I}_{D_9}/\mathcal{I}_{D_9}^2(3))=60$ gives us the expected number
of singular points.
We indeed check using Macaulay 2 that the two cubics containing the projection
of $V_9$
  constructed above are smooth and
intersect in a variety having 60 isolated singularities.
However, the singularities of $Y$ are not nodes but isolated triple points with
tangent cone being
 the cone over the projection of the cubic scroll to $\mathbb{P}^3$.
\end{rem}

\begin{rem}\label{codim one locus of projections}
One can describe the set of matrices defining
  those projections of $V_9$ that are
contained in two cubics in terms of maximal minors
of $L$ in $V$, hence, a variety of expected codimension~2.
On the other hand, from the proof of Proposition \ref{prop special Lambda}, for a
random choice of line $l$, the degeneracy locus restricted
to this line
  has two components over $\mathbb{Q}[\lambda]$. One component of
degree 60
  passes through our
  lift, and the other of degree 90 corresponds to
the locus of
  $\Lambda$'s which intersect the secant locus of $V_9$
(it is a simple exercise to prove that for any such $\Lambda$ the projection is
indeed contained in two cubics). This suggests that the locus of projections
satisfying the assertion of Proposition  \ref{prop special Lambda}
is an open subset of a hypersurface of degree 60 in $\mathbb{P}(V)$.
Additional evidence follows from the Jacobi formula for the derivative of the
determinant. Indeed, using this formula we can compute the tangent space to the
variety $\Gamma$
  at any constructed
  point, even if writing down the equations of
$\Gamma$ is out of reach for our computer. In our case, we get a codimension 1
tangent space.
It is an interesting problem to find a geometric interpretation as in
\cite{HulekRanestad} for the centers
of projection contained in $\Gamma$.

Another interesting problem is the geometric description of the cubics
containing the projected variety. For instance,
for a generic choice of
the center of projection $\Lambda$, the surface
$D_9\subset \mathbb{P}^5$ is contained in a unique cubic singular along a
  non-degenerate curve of degree $6$. From \cite{Laza}
  such a cubic has to be
determinantal.
\end{rem}

\begin{rem}
  If the center of projection $\Lambda$ intersects the secant locus of the surface $V_9$
in one
  point, then $D^{\Lambda}_9$ is also contained in a pencil of cubics. However, in this case, we have
$h^1(\mathcal{I}_{D^{\Lambda}_9})=3$. Moreover, if $\Lambda$ meets the secant locus in three (resp.~four)
points then $D^{\Lambda}_9$ is contained in a
  four- (resp.~five-) dimensional space of cubics. And if the intersection is a
  line, then there is
  a seven-dimensional space of cubics in the ideal.
\end{rem}
\subsection{Surfaces of general type of degree $18$ in $\mathbb{P}^5$}
\label{sec degree 18}
Let us consider a special $\Lambda$ such that $D^{\Lambda}_9\subset
\mathbb{P}^5$ (see Proposition \ref{prop special Lambda}) is contained in a
complete intersection threefold of degree $9$ with 60 isolated singularities as
above.
\begin{thm}\label{gen type surfaces of deg 18}
The surface $D^{\Lambda}_9$ can be bilinked through the complete intersection of two cubics to a smooth
surface of general type $S_0$ of degree $18$.
\end{thm}
\begin{proof} Let us denote by $H$ the class of the hyperplane on
$\mathbb{P}^5$.
Denote by $F_k$ the
surface residual to $D_9$ through a complete intersection of type
$(3,3,k)$ defined as the intersection of the cubics containing $D^{\Lambda}_9$ with a general hypersurface of degree $k$ for some $k\geq 4$.
  From the exact sequence $$0\to\mathcal{I}_{D_9\cup F_k}\to
\mathcal{I}_{F_k}\to \omega_{D_9}(-k)\to 0,$$ using the fact that
  $D_9\cup F_k$ is
  ACM and $\omega_{D_9}=\mathcal{O}_{D_9}(-1)$, we infer
  that $h^0(\mathcal{I}_{D_9\cup F_k}(k+1))+1=h^0(\mathcal{I}_{ F_k}(k+1))$. The
surface $F_k$ is thus linked through two cubics and a surface of degree $k+1$ to
a surface $S_0$.
Let $Y$ be the complete intersection of our cubics.
Using Macaulay 2 in characteristic $17$, we proved that the singular locus of $Y$
is a smooth
  zero-dimensional scheme of degree $60$. Moreover, in characteristic  17 this scheme is
also the intersection scheme of $D_9$ and $S_0$.
It follows that also in characteristic $0$ the surfaces $D_9$ and $S_0$
intersect in isolated points in a transversal
  way, i.e.~their tangent spaces
  intersect transversely at each intersection point.
Since $S_0$ and $D_9$ are contained in a smooth
  cubic, it follows that $S_0$ is
smooth at each intersection point.
To prove that $S_0$ is smooth everywhere, we use
  the Bertini theorem outside the
singular locus of $Y$. Indeed, observe first that by
  \cite[Prop.
4.4]{Hartshornebil} the variety $S_0$ is an almost Cartier generalized divisor
linearly equivalent to
the almost Cartier generalized divisor $D_9+H$. It follows that on $Y\setminus
\mathrm{Sing} (Y)$ the surface $S_0$ is a general element of the
 system
$|D_9+H|$ which is base point free on $Y\setminus \mathrm{Sing} (Y)$
because it is
  so outside $D_0$ and because $S_0$ does not meet
$D_9$ in
  $Y\setminus \textup{Sing} (Y)$.
Furthermore, by
  the adjunction formula, since $S_0$ is a smooth divisor, Cartier in
codimension 2, we have $K_{S_0}=(K_Y+D+H)|_{S_0}=H_{S_0}$. This means that $S_0$
is a surface of general
type, canonically embedded in $\mathbb{P}^5$.
It is clearly linearly
  normal, since by
  the basic properties of liaison its
Hartshorne--Rao module is the Hartshorne--Rao module of the del Pezzo surface with
gradation lifted by 1. Finally, the degree of $S_0$ is $18$ by construction.
\end{proof}
\begin{rem}
 Observe that if $E$ is the bundle defining the del Pezzo surface $D_9$ through
the Pfaffian construction, then $S_0$ is defined by the bundle $E\oplus
2\mathcal{O}_{\mathbb{P}^5}$. Indeed, let us first point out that since we are
dealing with almost Cartier divisors,
we can perform all computations on $Y\setminus \mathrm{Sing} (Y)$ and then
extend the result to $Y$. In particular, by Lemma \ref{lem bilinkage 1} we have
  a 6-dimensional
subsystem of the linear system
  $|D_9+H|$, consisting of varieties obtained as
sections of $E\oplus 2\mathcal{O}_{\mathbb{P}^5}$. To prove that this gives the
complete linear system
  $|D_9+H|$, we make a simple dimension count basing on the
exact
  sequence
$$0\to \mathcal{O}_Y(H) \to  \mathcal{O}_Y(D+H) \to \mathcal{O}_D(D+H) \to 0,$$
the fact that the singularities of $Y$ are normal of codimension 3, and the
equalities  $\mathcal{O}_D(D+H)=\mathcal{O}_D$
   and $h^1(Y,\mathcal{O}_Y(H) )=0$. We
obtain
 $h^0(Y,\mathcal{O}_Y(D+H))=h^0(Y, \mathcal{O}_Y(H))+h^0(D,
\mathcal{O}_D)=7$, hence
the system $|D_9+H|$ is of dimension 6.
\end{rem}

\begin{rem}
Observe that the dimension of the family of surfaces obtained in Theorem
\ref{gen type surfaces of deg 18} is
at most $20$. Indeed, since we know that the general choice of
  a center of
projection does not lead to a variety
contained in the complete intersection of two cubics, it follows that the
dimension of the space of $\Lambda$'s
up to linear automorphisms is at most $\dim G(4,10)-1-8=15$.
We also find that for a given $D_9\subset \mathbb{P}^5$ contained in the
intersection of the cubics there is a
$5$-dimensional family ($=h^0(D+H)-1$) of bilinked surfaces. Altogether this gives
a space of dimension at most $20$.
Moreover, by Remark \ref{codim one locus of projections} and the Betti table of
  the
constructed surface,
we expect this dimension to be exactly $20$.
\end{rem}

We know that the dimension of the Kuranishi space $\mathcal{K}$ of a surface of
general type $S$ is not smaller than $h^1(T_S)-h^2(T_S)$.
On the other hand,
  by Noether's formula we find $c_2(S)=66$. So, by
  the Riemann--Roch
theorem
applied to $T_S$, we infer that $h^1(T_S)-h^2(T_S)= 36+h^0(T_S)\geq 36$.
Since $h^1(T_{\mathbb{P}^5}|S)=0$,
  we also have
$H^0(N_{S|\mathbb{P}^5})\twoheadrightarrow H^1(T_S)$. Thus there should
be an at least $36$-dimensional family of canonically embedded surfaces of
general type of degree
$18$ in $\mathbb{P}^5$ with special element $S_0$. The general element of this
family should have a simpler Hartshorne--Rao module. It is a natural problem to
construct
  such a general Hartshorne--Rao module.

\subsection{Bilinkages of surfaces of general type} Let us now study how
  Construction \ref{relation by bilinkage and deformation surf} works in the case
of canonical surfaces in $\PP^5$ constructed in \cite{Catanese}.
\begin{prop}\label{lem can surfaces are bilinked to del pezzo}
Any smooth linearly normal canonical surface $S\subset \mathbb{P}^5$ of
degree $d_S\leq 17$ satisfying the maximal rank assumption is obtained by
Construction
\ref{relation by bilinkage and deformation surf} from a del Pezzo surface $D$ of
degree $d_D=d_S-9$ with bilinkage performed
in a complete intersection of two cubics.
\end{prop}
\begin{proof}
We compare the description, by \cite{Catanese}, of a canonical surface $S_{d}$ of degree
$12\leq d+9\leq 17$ in $\mathbb{P}^5$ satisfying the maximal rank assumption, with the description of a projected del Pezzo surface $D_{d}$ of degree $d$ contained in
\cite{CYTon1}. Since all surfaces $S_d$ satisfying the assumptions above are deformation equivalent, it is enough to obtain one such surface in each degree using Construction \ref{relation by bilinkage and deformation surf}
as in the assertion. We observe that for $d\leq 6$ the bundle $F_{d}$ constructed by Catanese is related to the corresponding bundle $E_{d}$ from \cite{CYTon1}
by $E_{d}\oplus 2\mathcal{O}_{\mathbb{P}^6}=F_{d}$.
Since for $d\leq 7$ the
del Pezzo surface $D_{d}$ is contained in a complete intersection of two cubic
hypersurfaces, by Lemma \ref{lem bilinkage 1}
the del Pezzo surface $D_{d}$ is bilinked to a Gorenstein surface of general type
defined by the bundle $E_{d}\oplus 2\mathcal{O}_{\mathbb{P}^6}$ through the
Pfaffian construction. Let us denote the general such surface by $\tilde{S_{d}}$.
Since, for $d\leq 6$, we have $E_{d}\oplus 2\mathcal{O}_{\mathbb{P}^6}=F_d$, it follows that
$\tilde{S}_{d}$ is a smooth canonical surface in $\mathbb{P}^5$ of degree $d+9$.

For $d=7$,
  both $E_7\oplus 2\mathcal{O}_{\mathbb{P}^5}$ and $F_7$  appear
as kernels of some surjective maps $13\mathcal{O}_{\mathbb{P}^5}\to
2\mathcal{O}_{\mathbb{P}^5}.$
Moreover $F_7$ is the kernel of a generic such map. It is now enough to take a
one-parameter family parametrized by $\lambda \in \mathbb{C}$ of maps as above such
that for $\lambda \neq 0$ the kernel is isomorphic to
  $F_7$, whereas for $\lambda =0 $ the kernel is $E_7\oplus
2\mathcal{O}_{\mathbb{P}^5}$.
It follows that there is a bundle $\mathfrak{E}$ on
$\mathbb{P}^6\times \mathbb{C}$ whose restriction to $\mathbb{P}^6\times \{0\}$
is the bundle $\mathcal{F}\oplus 2\mathcal{O}_{\mathbb{P}^6}$ and the
restriction to
a fiber $\mathbb{P}^6\times \{\lambda \}$ for $\lambda\neq 0$
 is isomorphic to $E$. Moreover, since  $h^0(\bigwedge^2((E_7\oplus
2\mathcal{O}_{P}^5)(1))=h^0(\bigwedge^2 F_7(1))$,
 we infer that each section of
$\bigwedge^2(E_7\oplus2\mathcal{O}_{\mathbb{P}^6}) (1)$ is
extendable to a section of $\bigwedge^2\mathfrak{E}(1)$.
 It follows
  from \cite[Lem.~3.4]{CYTon1} that $\tilde{S}_7$ is a degeneration
of a family of canonical surfaces of degree 16 satisfying the maximal rank
  assumption; since all such surfaces are deformation
  equivalent, the assertion
follows in the case $d=7$.

For $d=8$ the situation is similar. More precisely, let $D_8$ be a del Pezzo
surface of degree 8 (either of type  $D_8^1$ or $D_8^2$) defined by some bundle
$E_8$. Then $D$ is contained in a variety
  that is the complete intersection of two
  cubics, and
$E_8\oplus 2\mathcal{O}_{\mathbb{P}^5}$ is the kernel of some special surjective
map
$16 \mathcal{O}_{\mathbb{P}^5}\to 3\mathcal{O}_{\mathbb{P}^5}.$
The bundle $F_8$
  which is the kernel of a generic such map defines by \cite{Catanese} a
canonical surface of general type satisfying the maximal rank assumption.
We then construct $\mathfrak{E}$ in the same way as above and
  conclude the proof by applying \cite[Lem.~3.4]{CYTon1} and the equality  $h^0(\bigwedge^2((E_8\oplus
2\mathcal{O}_{P}^5)(1))=h^0(\bigwedge^2 F_8(1))$.
\end{proof}
\section{Del Pezzo threefolds and Tonoli Calabi--Yau threefolds}\label{sec-Ton}
In order to obtain a Calabi--Yau threefold by
  Construction \ref{relation by
bilinkage and deformation surf}, we consider del Pezzo threefolds
  $T$, i.e.~$K_{T}=-2H$, where $H$ is ample embedded
in $\mathbb{P}^6$ by a subsystem of the half-anticanonical class. The first
examples of such threefolds we consider are the del Pezzo threefolds
of degree $3\leq d\leq 5$ embedded by the complete linear system of the half-anticanonical class into a linear subspace of $\PP^6$.
Note that all the del Pezzo threefolds are enumerated in \cite[Thm.~3.3.1]{Isk}.

Recall that, for a threefold in $Y\subset \mathbb{P}^n$ with $n\geq 7$ which is
not contained in a hyperplane,
there exists a smooth projection of $Y$ into $\mathbb{P}^6$ if and only if $Y$
is
  defective, i.e.  the dimension of the secant locus of $Y\subset \PP^n$ is $\leq 6$.
It follows that if we want to consider only smooth threefolds
  $F\subset \mathbb{P}^6$ in the case $d\geq 6$, we are restricted to the consideration of
del Pezzo threefolds
defective in their half-anticanonical embedding. There are only a few examples
of such. Let $V_t\subset \PP^{t+1}$ with
$t = 7, 8$
be the image of $\PP^3$ by the map defined by all
quadrics passing through $8-t$ point in $\PP^3$.
\begin{lemm}\label{FanoD}
 A smooth del Pezzo threefold of degree $d\geq 6$ embedded in $\PP^6$
by a subsystem of the half-anticanonical embedding is defective if and only if it is isomorphic to one of the following:
 \begin{enumerate}
 \item  $T_6\subset \PP^6$ is the generic central projection of the hyperplane section of the Segre embedding of
$\mathbb{P}^2\times \mathbb{P}^2\subset \PP^8$;
 \item $T_7\subset \PP^6$ is the projection of $V_7\subset \PP^8$
 from a linear
 space disjoint from the secant locus $Sec(V_7)\subset \PP^8$;
 \item $T_8\subset \PP^6$ is the  projection of the image $V_8$ of the double
Veronese embedding of $\mathbb{P}^3$ into $ \mathbb{P}^9$ from a linear
 space disjoint from the secant locus $Sec(V_8)\subset \PP^9$;
 \end{enumerate}
\end{lemm}
\begin{proof}
   This follows by comparing the classical classification of defective threefolds due to Scorza (see \cite{Scorza} for a modern approach) and the classification of del Pezzo threefolds due to Iskovskikh
 (see \cite{IskovskikhFano}).
\end{proof}

We aim at proving the following.
\begin{prop}\label{prop bilinkage dP 7 and CY 16}
For $6\leq d\leq 7$ every smooth del Pezzo threefold $T_d$ of degree $d$ in
$\mathbb{P}^6$ is related by
  Construction \ref{relation by bilinkage and
deformation surf} to a smooth Calabi--Yau threefold of degree $d+9$ in
$\mathbb{P}^6$.
\end{prop}
\begin{proof} Recall that del Pezzo threefolds are defined through the Pfaffian
construction by the bundles $E_d'$ from \cite[Rem.~4.6]{CYTon1}.
 Since for $d\leq 7$ the del Pezzo
threefold $T_d$ is contained in a complete intersection of two cubic
hypersurfaces, by Lemma \ref{lem bilinkage 1},
the del Pezzo threefold of degree $d$ is bilinked to a Gorenstein Calabi--Yau threefold (a Gorenstein threefold with $\omega_X=0$ and $h^1(X, \mathcal{O}_X)=h^2(X, \mathcal{O}_X)=0$)
$\tilde{X_d}$ defined by the bundle $E'_d\oplus
2\mathcal{O}_{\mathbb{P}^6}$ through the Pfaffian construction. Now for $d\leq
6$ we have
$E'_d\oplus 2\mathcal{O}_{\mathbb{P}^6}=F_d$ and hence $\tilde{X}_d$ is a smooth
Tonoli Calabi--Yau threefold of degree $d+9$. For $d= 7$ it is enough to
observe that there is a bundle $\mathfrak{E}$ on
$\mathbb{P}^6\times \mathbb{C}$ whose restriction to $\mathbb{P}^6\times \{0\}$
is
   $E'_7\oplus 2\mathcal{O}_{\mathbb{P}^6}$ and the restriction to
any fiber $\mathbb{P}^6\times \{\lambda \}$ for $\lambda\neq 0$ is isomorphic to
$F_7$. We also compute using Macaulay 2 that the dimension of the space
 of sections
  of  $\bigwedge^2(E'_7\oplus2\mathcal{O}_{\mathbb{P}^6})(1)$ is equal
to the dimension of the space of sections
 of $\bigwedge^2 F_7(1)$. We infer that
$\bigwedge^2(E'_d\oplus2\mathcal{O}_{\mathbb{P}^6}) (1)$ is
extendable to a section of $\bigwedge^2\mathfrak{E}(1)$.
 It follows
  from \cite[Lem.~3.4]{CYTon1} that $\tilde{X}_d$ is a
degeneration of a family of Tonoli Calabi--Yau threefolds.
\end{proof}
\begin{rem} Note that the above proof works for each del Pezzo threefold od degree $3\leq d\leq 5$ such that the half-anticanonical divisor gives an embedding into a linear subspace of $\PP^6$.
  In this way we obtain all smooth ACM Calabi--Yau threefolds.
\end{rem}
Let now study the most interesting case:
  let $T_8\subset \PP^6$ be a del Pezzo threefold of degree $8$ in $\mathbb{P}^6$
  which is the projection of $V_8\subset \PP^9$ as above.
Using the methods from \cite{HanKwak} we deduce that the ideal of $T_8\subset
\mathbb{P}^6$ is generated by $45$ quartics and is not contained in any cubic.

However, in order to perform a bilinkage we can find a special center of
projection $L\subset \PP^9$ also disjoint from the secant locus $Sec(V_8)\subset \PP^9$,
such that the image
 of the projection $T^L_8\subset \PP^6$ is contained in a $3$-dimensional system of cubics.
\begin{prop}\label{bilink dla fano st 8}
 There exists a center of projection $L\subset \PP^9$ such that $T^L_8\subset\mathbb{P}^6$ can be
bilinked to a Gorenstein Calabi--Yau threefold (not necessarily normal) $X'\subset \PP^6$ of
degree 17. Moreover, one can choose the bilinkage in such a way that $X'\subset \PP^6$ admits
a smoothing by the family of Tonoli Calabi--Yau threefolds of degree $17$ with
$k=9$.
\end{prop}
\begin{proof}
Recall that the $2\times 2$ minors of the
   matrix
$$A=\left[\begin{array}{cccc}
a&x&y&z\\x&b&t&u\\y&t&c&v\\ z&u&v&w
 \end{array}\right].$$
define the second Veronese embedding of $\mathbb{P}^3$ in
$\mathbb{P}^9(a,b,c,x,y,z,t,u,v,w)$.
Let us consider a special
  $\Lambda=\mathbb{P}^2$, the center of projection
defined by the following equations:
$$
 \left\{ \begin{array}{c} -2a+b+c-2y-z+2t+2u+2v-w=0\\ 2a+2b+c+x+y-z-t+2u-2v-w=0 \\
-a-2b-2c-2x+y+t+v+w=0\\ 2a+b-2c+2y-2z+2u-v+w=0\\
a+b-2z+2t+u+2v-w=0\\ -2c-2x+y-z-t+2u+w=0\\ -c+2x-y-2u-v+2w=0
           \end{array}
\right .$$
  Although it is hard to check that by hand, it is straightforward to verify using
Macaulay 2
that the image of the projection $T_8^L\subset \mathbb{P}^6$ is contained in three
independent cubics. Then the residual to $T_8^L$ of the intersection of these cubics
is a threefold $G$ of degree $19$ that is contained
in a quartic that does not contain $T_8^L$ (cf. \cite[Lemma
3.1]{GKapustkaprojdelpezzo}). The residual to
$G$ in the intersection of the cubics with the quartic containing $G$
is a threefold $X'$ of degree $17$ (we say that $X'$ is bilinked with $T_8^L$
through two cubics with height $1$).

As in
  Propositions \ref{lem can surfaces are bilinked to del pezzo} and
\ref{prop bilinkage dP 7 and CY 16},
by Lemma \ref{lem bilinkage 1}, we infer that there is a
Pfaffian variety $X$ associated to the vector bundle $E\oplus
2\mathcal{O}_{\mathbb{P}^6}$ which is a Gorenstein Calabi--Yau threefold of
degree 17.

For the second part, note that, by the general properties of bilinkage, the
threefold $X$ has the same
Hartshorne--Rao module as $T_8^L$ but shifted by one. It follows by \cite[proof of Thm. 1.3]{CYTon1} that this
Hartshorne--Rao module is determined by some special
 $\mathbb{P}^{13}\subset \langle\mathbb{P}^2\times \mathbb{P}^6\rangle$ containing a linear
space $\mathcal{P}$ spanned by the graph of some double Veronese embedding
(composed with a linear embedding) of $\mathbb{P}^2$ to $\mathbb{P}^6$.
 Observe moreover that, by \cite[Thm.~1.1]{CYTon1}, the Hartshorne--Rao module of a Tonoli Calabi--Yau
threefold of degree
17 with $k=9$ corresponds to a $\mathbb{P}^{15}\subset \langle\mathbb{P}^2\times
\mathbb{P}^6\rangle$ containing such a linear space $\mathcal{P}$. We now claim
that
   the bundle
$E\oplus2\mathcal{O}_{\mathbb{P}^6}$ appears as a flat deformation of a family of bundles
associated to such Calabi--Yau threefolds.
Indeed, the bundle $E\oplus2\mathcal{O}_{\mathbb{P}^6}$ is obtained as the kernel of a map
$16\mathcal{O}_{\mathbb{P}^6}\to 3\mathcal{O}_{\mathbb{P}^6}(1)$
defined by a matrix whose columns
span the $\mathbb{P}^{13}$, whereas the chosen Tonoli Calabi--Yau threefolds
appear as a Pfaffian variety associated to a bundle
obtained as the kernel of a similar
  map, but with columns spanning a
$\mathbb{P}^{15}$ containing our $\mathbb{P}^{13}$. It is easy to see that by
degenerating two columns of the map to zero (for example by multiplying them by
the parameter
$\lambda$) one
  obtains the desired flat deformation.

Observe now that there exists a subspace $V$ of dimension 9 of the space of
sections
  of $\bigwedge^2(E\oplus2\mathcal{O})(1)$,
consisting of
  all sections which admit extensions to our deformation family.
By \cite[Lem.~3.4]{CYTon1} the varieties given by these sections admit
smoothings to Tonoli Calabi--Yau threefolds of degree 17 with $k=9$.
\end{proof}
\begin{rem} \label{two components 1}Observe that, in the proof of Proposition \ref{bilink dla fano st 8},
not every section of $\bigwedge^2(E\oplus2\mathcal{O}_{\mathbb{P}^6})(1)$ extends to the
deformation family.
It follows that taking a general section of
$\bigwedge^2(E\oplus2\mathcal{O}_{\mathbb{P}^6})(1)$ in the proof of Proposition \ref{bilink
dla fano st 8}
one obtains a Gorenstein Calabi--Yau threefold representing
a different component of the Hilbert scheme of Calabi--Yau
threefolds consisting possibly of only singular threefolds.
\end{rem}
\section{Unprojections}\label{sec-unpr}
Recall that \textit{unprojection} is the inverse process to projection (see \cite{ReidPapadakis} for a general discussion).
  In this section we discuss  the relations between the constructions by  unprojection and by bilinkage in the context of submanifolds of codimension $3$.

For the construction of Calabi--Yau threefolds using bilinkages with del Pezzo
threefolds we are not restricted to
  starting from smooth Fano threefolds. A
natural choice for singular del Pezzo threefolds are cones over del Pezzo
surfaces. These are always contained in many
cubics and a bilinkage can be performed. This construction
  enables one to directly
relate the del Pezzo surface
  to the Calabi--Yau threefolds constructed.
When the cone is
  Gorenstein, it is related to so-called Kustin--Miller
unprojections.
This construction was studied in
\cite{Cascade,Unprojections,GKapustkaprojdelpezzo}.
In particular, a straightforward generalization of
\cite[Prop.~4.1]{Unprojections} (cf.~\cite{BoehmPapadakis}) shows that the unprojection of a codimension $3$ variety
defined by
Pfaffians of a decomposable bundle $E$ on $\mathbb{P}^n$
in a codimension 2 complete intersection can be seen as
  some special Pfaffian
variety associated to the bundle $E'\oplus
\mathcal{O}_{\mathbb{P}^{n+1}}(a_1)\oplus \mathcal{O}_{\mathbb{P}^{n+1}}(a_2)$
  where $E'$ denotes the
trivial extension of the decomposable bundle $E$ to $\mathbb{P}^{n+1}$, and $a_1$
and $a_2$ are appropriate numbers depending on the degrees of the generators of
the complete intersection and the degrees in the decomposition of $E$.

In the case of a complete intersection of two cubics containing a projectively
Gorenstein del
Pezzo surface in $\mathbb{P}^5$, the result of the unprojection is a special
Pfaffian variety associated to $F=E'\oplus 2\mathcal{O}_{\mathbb{P}^6}$. More
precisely, it is given as
the degeneracy locus of a skew-symmetric map $\rho: F^*\to F\otimes
\mathcal{O}_{\mathbb{P}^6}(1)$ corresponding to a section of the form
$$(\varphi,c_1,c_2,x_6)\in H^0({\textstyle\bigwedge}^2 F\otimes
\mathcal{O}_{\mathbb{P}^6}(1))=H^0({\textstyle\bigwedge}^2 E'\otimes
\mathcal{O}_{\mathbb{P}^6}(1))\oplus 2H^0(E'(1))\oplus
H^0(\mathcal{O}_{\mathbb{P}^6}(1)),$$
where $\varphi$ defines the cone over the del Pezzo surface,  $c_1,c_2$ are
sections which correspond via the Pfaffian resolution to two cubics containing
the del Pezzo surface, and $x_6$ is the new variable of $\mathbb{P}^6$.
The following follows
\begin{cor}
Every Tonoli Calabi--Yau threefold of degree $\leq 14$ is a smoothing of a
Gorenstein Calabi--Yau threefold
obtained as the unprojection of a del Pezzo surface of degree $d \leq 5$ in a
complete intersection of two cubics.
\end{cor}

In the case $d\geq 6$ a standard Kustin--Miller unprojection cannot be performed
because the del Pezzo surface is not projectively Gorenstein.
This
   is the first case in which the cone over the del Pezzo surface is not
Gorenstein
  at its vertex and, as such, it cannot be written in terms of the Pfaffian construction applied to a
vector bundle.
We can, however,
  somehow ignore this fact and propose a non-Gorenstein
unprojection instead of the standard construction due to Kustin and Miller.
  More precisely, by a non-Gorenstein unprojection we mean that having a variety
$D\subset Y\subset \mathbb{P}^N$ with $D$ not projectively Gorenstein we
construct
a variety $X\subset \mathbb{P}^{N+1}$ singular
  at some point $p$ such that the
projection of $X$ from $p$ is $Y$ and the exceptional locus is $D$.
\begin{prop} \label{subsection unprojection of del pezzo of degree 6}
 A Tonoli Calabi--Yau threefold of degree 15 can be obtained as a smoothing of a
singular variety obtained as a non-Gorenstein unprojection
 of a del Pezzo surface of degree $d=6$ in a complete intersection of two
cubics.
\end{prop}
\begin{proof}  Observe that, although the cone over the del Pezzo surface $D_6$
is not Gorenstein, we have its description in terms of some similar
Pfaffian construction applied to the
  sheaf $E'$,
trivially extending $E$ to $\mathbb{P}^6$. In this case the special Pfaffian
variety associated to the sheaf $F=E'\oplus 2\mathcal{O}_{\mathbb{P}^6}$
obtained by copying the unprojection procedure above in the context of sheaves
is a non-Gorenstein variety $X$. We shall prove that it admits a smoothing to
a Calabi--Yau threefold of degree $15$.
More precisely, we proceed in the following way.
We start with a del Pezzo surface $D_6$. It is obtained as a Pfaffian variety
associated to the bundle $E=\Omega^1_{\mathbb{P}^5}(1)\oplus 2
\mathcal{O}_{\mathbb{P}^5}$, i.e. defined as the degeneracy locus of a general
skew-symmetric map
$\phi:E^*(-1)\to E$. We consider two cubics in the ideal of the del Pezzo
surface.
From the Pfaffian
sequence they correspond to two sections of $E(1)$ giving a map
$\psi:2\mathcal{O}_{\mathbb{P}^5}(-1)\to E$. We can now extend the bundle $E$ to a sheaf
$E'$ on $\mathbb{P}^6$ defined as the kernel of the map
$8\mathcal{O}_{\mathbb{P}^6} \to \mathcal{O}_{\mathbb{P}^6}(1)$ given by
the matrix $[x_0,\dots,x_5,0,0]$.
Then we consider the
  skew-symmetric map $\rho:(E'\oplus
2\mathcal{O}_{\mathbb{P}^6})^{\vee}(-1) \to  E'\oplus
2\mathcal{O}_{\mathbb{P}^6}$ defined by $\phi,\psi$ and multiplication by the
new variable $x_6$.
The degeneracy locus of $\rho$ is a codimension $3$ variety $X'$ which is
singular
  at
the point $(x_0,\dots,x_6)=(0,\dots,0,1)$, the tangent cone being the cone over
the projected del
Pezzo surface $D_6$. The latter singularity is not Gorenstein.
  Hence our
variety cannot be described as a Pfaffian variety associated to a vector bundle
(we have its description as a kind of
Pfaffian variety associated to the sheaf $E'$). It is however straightforward to
check that the projection
  from the point
$(0,\dots, 0,1)\in \mathbb{P}^6$ maps
$X'$ to the complete intersection of the two cubics containing the del Pezzo
surface,
and the exceptional locus is $D_6$.

Having the description of $X'$ in terms of Pfaffians (associated to a sheaf), we
perform a similar reasoning as in
\cite[Prop.~7.2]{CYP6} and prove that
  $X'$, though  not Gorenstein and not
  normal, can nonetheless be smoothed to a
Tonoli Calabi--Yau threefold of degree 15. More precisely, following \cite{Catanese} we
can consider $\rho$ as a
$10 \times 10$ skew-symmetric matrix $A$ of linear forms satisfying the
the following
$$[x_0\dots x_5,0,0,0,0]\cdot A=0.$$
The degeneracy locus of $\rho$ is given by $8\times 8 $ Pfaffians of $A$.
Observe that by the shape of unprojection and the assumption on $A$ we can write
$A$ in the
  form
\[
\left(
\begin{BMAT}(rc){c|c|c}{c|c|c}
B & \begin{BMAT}(rc){c}{ccc}
        a_0\\
        \vdots\\
        a_5
    \end{BMAT}&D^T\\
\begin{BMAT}(rc){ccc}{c}
-a_0&\dots&-a_5
\end{BMAT} & 0 &\begin{BMAT}(rc){ccc}{c}a_7&a_8&a_9\end{BMAT}
\\
D&\begin{BMAT}(rc){c}{ccc}
        -a_7\\
        -a_8\\
        -a_9
    \end{BMAT}&K\end{BMAT}
\right),
\]
where the variable $x_6$ appears only in the matrix $K$
  (more precisely, in a
$2\times 2$ skew-symmetric submatrix of $K$).
Since $[a_0\dots a_5]$ satisfies a Koszul relation, there exists a skew-symmetric
$5\times 5$ matrix $B'$ with complex entries such that
$$\left[\begin{matrix} a_0\\ \vdots\\a_5 \end{matrix} \right]=B'\cdot
\left[\begin{matrix}x_0\\ \vdots\\x_5 \end{matrix} \right].$$
Moreover since $a_i$ do not depend on $x_6$, there is clearly a unique $3\times
6$ matrix $D'$ with complex entries such that
$$\left[\begin{matrix}a_7\\ a_8 \\a_9 \end{matrix} \right]=D'\cdot
\left[\begin{matrix}x_0\\ \vdots\\x_5 \end{matrix} \right].$$

Consider now the family of skew-symmetric
  matrices
\[A_{\lambda}=
\left(
\begin{BMAT}(rc){c|c|c}{c|c|c}
B +\lambda x_6 B'& \begin{BMAT}(rc){c}{ccc}
        a_0\\
        \vdots\\
        a_5
    \end{BMAT}&D^T+(\lambda x_6 D')^T\\
\begin{BMAT}(rc){ccc}{c}
-a_0&\dots&-a_5
\end{BMAT} & 0 &\begin{BMAT}(rc){ccc}{c}a_7&a_8&a_9\end{BMAT}
\\
D+\lambda x_6 D'&\begin{BMAT}(rc){c}{ccc}
        -a_7\\
        -a_8\\
        -a_9
    \end{BMAT}&K\end{BMAT}
\right),
\]
parametrized by $\lambda\in \mathbb{C}$. Observe that in this case
$[x_0,\dots
x_5, \lambda x_6, 0,0,0]\cdot A_{\lambda}=0$. Hence the matrices $A_{\lambda}$
induce sections of $\bigwedge^2 E_{\lambda}(1)$, with $E_\lambda$ isomorphic to
$\Omega^1_{\mathbb{P}^6}(1)\oplus 3
\mathcal{O}_{\mathbb{P}^6}$, and the ideals generated by their $8\times 8$
Pfaffians correspond to Pfaffian varieties associated to $E_{\lambda}$.
To finish the proof, it is enough to observe that the
  above family is flat
around $\lambda=0$.
\end{proof}

\begin{rem}\label{two components}
Observe that
  we have used the special form of the section
defining the unprojected variety. In particular the construction could not be
performed if we were unable to find a matrix $A$ with
  all the $a_i$
for $i \in \{7,8,9\}$ independent of $x_6$. This suggests that the Hilbert
scheme of Calabi--Yau threefolds of degree $15$ has at least two components: one
giving the Tonoli family of degree $15$; the other parametrizing a family of
  non-Gorenstein threefolds probably birational to the degree $15$ threefolds in
$\mathbb{P}(1,1,1,1,1,1,1,2)$ constructed in \cite{GKapustkaprojdelpezzo}.
  If that is indeed the case, the unprojected threefolds $X'$ above would correspond to some
points in the intersection of these two components.
\end{rem}

One can try to extend the construction from the case of $d=6$ to higher degree
del Pezzo surfaces.
For instance,
   as in the proof of Proposition \ref{bilink dla fano st
8}, the Hartshorne--Rao modules of the cones over $D^1_8$ and
$D^2_8$ are degenerations of Hartshorne--Rao modules
associated to Tonoli Calabi--Yau threefolds of degree $17$ and $k=9$, $11$
respectively. The sheafified first syzygy modules of their
Hartshorne--Rao modules are not
  vector bundles, but more general sheaves. However, one
can still hope that,
as in the case of degree $d=6$, these non-Gorenstein threefolds admit smoothing
to Calabi--Yau threefolds.
Proceeding further, we compute the dimension of the space of sections of the
twisted second wedge power corresponding to the unprojection
and in each case we obtain a bigger space than the space of sections of
  the second
wedge power of the bundle defining the appropriate families of Tonoli
Calabi--Yau threefolds.
  Thus  again (cf. Remarks \ref{two components 1}, \ref{two components}) we obtain distinct components of the
Hilbert scheme of Calabi--Yau threefolds of degree 17 in $\mathbb{P}^6$. The
smoothing might possibly be performed only for very special unprojections.
It is also not clear whether the varieties representing the general points of
any of these components are smooth Calabi--Yau
  threefolds.

\subsection{Calabi--Yau threefolds of degree 18 via unprojection}
 Using the
method of unprojection,  we can also construct a
non-Gorenstein projective threefold
with one singular point with singularity locally isomorphic to the cone over a
projected del Pezzo surface of degree $9$. More precisely, let us start with a
del
Pezzo surface
$D^\Lambda_9$ from Proposition \ref{prop special Lambda}. It is contained in a
complete intersection $Y$ of two cubic hypersurfaces. Let $E$ be the vector
bundle on $\mathbb{P}^5$
defining $D^\Lambda_9$. Consider the non-Gorenstein unprojection of
$D^\Lambda_9$
in $Y$, i.e. a threefold $X$ defined as the degeneracy locus of a special
skew-symmetric map between the sheaf
$E'\oplus 2\mathcal{O}_{\mathbb{P}^6}$ and its twisted dual, as in the case of
degree $d=6$. Here, $E'$ is the sheaf on $\mathbb{P}^6$ obtained as the trivial
extension of $E$.
In this case, $X$ is a threefold with one singular point such that the projection
  from this point is $Y$ and the exceptional locus is $D^\Lambda_9$. Moreover, $X$
has degree $18$ and is birational to a Calabi--Yau threefold. Unfortunately, $X$
has no smoothing.
\bigskip

\bibliographystyle{alpha}
\bibliography{biblio}

\begin{minipage}{15cm}
 Department of Mathematics and Informatics,\\ Jagiellonian
University, {\L}ojasiewicza 6, 30-348 Krak\'{o}w, Poland.\\
\end{minipage}

\begin{minipage}{15cm}
Institute of Mathematics of the Polish Academy of Sciences,\\
ul. \'{S}niadeckich 8, P.O. Box 21, 00-956 Warszawa, Poland.\\
\end{minipage}

\begin{minipage}{15cm}
Institut f\"ur Mathematik\\
Mathematisch-naturwissenschaftliche Fakult\"at\\
Universit\"at Z\"urich, Winterthurerstrasse 190, CH-8057 Z\"urich\\
\end{minipage}

\begin{minipage}{15cm}
\emph{E-mail address:} grzegorz.kapustka@uj.edu.pl\\
\emph{E-mail address:} michal.kapustka@uj.edu.pl
\end{minipage}

\end{document}